  \documentclass[12pt]{amsart}
\address{Faculty of mathematics \\ NRU HSE \\ Moscow, Usacheva str.,6,  119048}

  \email{vasirog[at]gmail.com}

\usepackage[utf8]{inputenc}
\usepackage[T2A]{fontenc}
\usepackage{amssymb, amsmath, amsthm}
\usepackage{amscd}
\usepackage{graphicx}
\usepackage{hyperref}
\hypersetup{
    colorlinks=true,
    linkcolor=blue,
    filecolor=cyan,      
    urlcolor=cyan,
}
 
\usepackage[a4paper,hmargin=2.5cm,vmargin=2.5cm]{geometry}
\usepackage[english]{babel}
\graphicspath{ {images/} }
\usepackage{wrapfig}
\usepackage[matrix,arrow,curve]{xy}\usepackage{mathabx}

\usepackage{color}

\begin{document}
\newtheorem{predl}{Proposition}[section]
\newtheorem{thrm}[predl]{Theorem}
\newtheorem{lemma}[predl]{Lemma}
\newtheorem{cor}[predl]{Corollary}
\theoremstyle{definition}
\newtheorem{df}{Definition}
\renewcommand{\proofname}{\textnormal{\textbf{Proof:  }}}
\newcommand{\rmk}{\textnormal{\textbf{Remark:  }}}
\renewcommand{\C}{\mathbb C}
\renewcommand{\refname}{Bibliography}
\newcommand{\Z}{\mathbb Z}
\newcommand{\Q}{\mathbb Q}
\renewcommand{\o}{\otimes}
\newcommand{\e}{\mathcal{P}}
\newcommand{\R}{\mathbb R}
\renewcommand{\O}{\mathcal O}
\newcommand{\g}{\mathfrak g}
\newcommand{\D}{\mathcal D}
\renewcommand{\i}{\sqrt{-1}}
\renewcommand{\Im}{\operatorname{Im}}
\newcommand{\Ker}{\operatorname{Ker}}
\newcommand{\Hom}{\operatorname{Hom}}
\newcommand{\Alb}{\operatorname{Alb}}
\newcommand{\Hilb}{\operatorname{Hilb}}
\newcommand{\CP}{\mathbb{C}\mathbf{P}}
\newcommand{\acts}{\lefttorightarrow}
\newcommand{\codim}{\operatorname{codim}}
\newcommand{\hdot}{{\:\raisebox{3pt}{\text{\circle*{1.5}}}}}
\newcommand{\Tw}{\operatorname{Tw}}
\newcommand{\tw}{\mathrm{tw}}
\binoppenalty = 10000
\relpenalty = 10000

\title{Non-algebraic deformations of flat K\"ahler manifolds}

\author{Vasily Rogov} \thanks{Affiliation: National Research University --- Higher School of Economics / Humboldt Univerist\"at zu Berlin ; Email: vasirog@gmail.com. The study has been  partially funded within the framework of the HSE University Basic Research Program and the Russian Academic Excellence Project '5-100.}

\begin{abstract}
Let $X$ be a compact K\"ahler manifold with vanishing Riemann curvature. We prove that there exists a manifold $X'$ deformation equivalent to $X$ which is not an analytification of any projective variety, if and only if $H^0(X, \Omega^2_X) \neq 0$. Using this, we recover a recent theorem of Catanese and Demleitner, which states that a rigid smooth quotient of a complex torus is always projective.

We also produce many examples of non-algebraic flat K\"ahler manifolds with vanishing first Betti number.
\end{abstract}

\maketitle

\tableofcontents

\section{Introduction and preliminaries}

Any smooth projective variety over $\C$ can be viewed as a compact K\"ahler manifold. 

 However, the world of K\"ahler geometry is  noticeably  larger than the world of complex projective geometry. There exist holomorphic deformations of  projective varieties  which cannot be holomorphically embedded into any complex projective space. Moreover, Claire Voisin has constructed  examples of compact K\"ahler manifolds, which are not homeomorphic to any projective variety (\cite{Vois04}).

Let us say that a complex structure is {\it algebraic}, or, more generally, that a complex manifold is {\it algebraic}, if it can be obtained as analytification of a smooth projective variety. The list of manifolds that admit both algebraic and non-algebraic complex structures includes K3-surfaces, Hyperk\"ahler manifolds, complex tori and others. For complex tori, it is well-known that in every dimension $n>1$ a very general $n$-dimensional complex torus is non-algebraic (\cite{BL99}, Ch. 1, Corollary 6.3. See also \cite{Sh}, Ch. VII, Subsection 1.4.). As it follows from Barlet's theorems about  the semi-continuity property for the algebraic dimension (\cite{Bar}), if there exists at least one non-algebraic complex structure in a given deformation class, then a very general member of this family is non-algebraic.

The aim of this article is to study  the existence of non-algebraic deformations of  complex  manifolds admitting flat K\"ahler metric. Such manifolds are natural generalisations of complex tori, and, as follows from Bieberbach's theory of crystallographic groups (see  subsection~\ref{Fedorov}), every flat K\"ahler manifold $X$ is isomorphic to a quotient $T/G$ for a complex torus $T$ and a finite group $G$ acting on $T$ holomorphically and freely.

Observe that in this case $H^{\hdot}(X, \Q) = H^{\hdot}(T, \Q)^G$ (see e.g. \cite{Hat}, Prop. 3G.1 on p. 321) and the isomorphism is compatible with Hodge decomposition (\cite{Vois02}, Section 7.3.2).

 In general, there might be no non-trivial holomorphic $2$-forms on $X$. This gives an obvious obstruction for $X$ to admit a non-algebraic deformation. If $H^0(X, \Omega^2_X) = H^{2,0}(X)=0$, the same holds for any deformation $X'$ of $X$, since Hodge numbers are constant in flat smooth families.  The cone of K\"ahler classes is open inside $$H^{2}(X', \R)  = H^2(X', \R) \cap H^{1,1}(X')$$ and we are able to find a rational K\"ahler class.  By Kodaira's embedding theorem this realises $X'$ as a subvariety in $\CP^N$.

In  Theorem \ref{deformations} we show that this is indeed the only obstruction for a flat K\"ahler manifold to admit a non-algebraic deformation: if $H^0(X, \Omega^2) \neq 0$, there exists a family of flat K\"ahler manifolds over a disc (in fact, over a projective line) with special fibre being isomorphic to $X$ and very general fibre being non-algebraic (see  Section \ref{main section} for the precise statement).

We would also like to mention another motivation for our work, which  comes from   the Amerik-Rovinsky-Van de Ven conjecture. In their article (\cite{ARVdV}, Section 3) these authors conjectured that  the second Betti number of any compact  flat K\"ahler manifold is at least two. Observe that any compact K\"ahler manifold has $b_2 > 0$ and if $b_2 = 1$,  this manifold automatically admits a K\"ahler form with integral cohomology class and therefore is algebraic. There is a recent preprint of R. Lutowski  \cite{L}, claiming to prove a statement equivalent to ARVdV Conjecture. The work of Lutowski  is sufficiently based on an  earlier paper \cite{HS}, which uses the classification of simple finite groups. In spite of this, we are still hoping for some advances in understanding this problem from a more geometric point of view.

The paper is organised in the following way: first, we explain the main properties of compact flat K\"ahler manifolds and give some examples of those. In Section \ref{main section} we give the precise formulation for our main theorem (Theorem~\ref{deformations}). Then we recall some heuristics coming from Hyperk\"ahler geometry, prove a pair of lemmas from linear algebra, and deduce the main theorem. As a corollary, we obtain a new proof of the Catanese-Demleitner theorem (\cite{CD+C}, Theorem 1).

Finally, in Section \ref{quaternionic doubles} we describe examples of non-algebraic flat K\"ahler manifolds and provide a construction for flat K\"ahler manifolds which admit holomorphic symplectic form and have vanishing first Betti number. Together with the main theorem, this gives many examples of non-algebraic flat K\"ahler manifolds with $b_1=0$. This is of independent interest since explicit examples of such manifolds do not seem to be present in the literature.

\section*{Acknowledgements} I am thankful to  Misha Verbitsky for turning my attention to this subject, for his support and  for fruitful conversations during the preparation of this paper. I am  thankful to  Ekaterina Amerik, Rodion D\'eev and Roman Krutowski  for several useful discussions and to the anonymous reviewer from Mathematical Research Letters journal for important comments concerning both the form and the content of this work. I am also thankful to Alejandro Tolcachier for a remark regarding the preliminary version of this text.

\subsection{Flat K\"ahler manifolds}\label{Fedorov}

Recall that {\it a crystallographic group} is a discrete subgroup $\Gamma \subset \operatorname{Iso}(\R^n)$ of isometries of  Euclidean space, such that  $\R^n/\Gamma$ is compact. The study of such groups was initiated by Evgraf Fedorov \cite{Fed}. The modern theory of crystallographic groups is based on the following theorem of Ludwig Bieberbach(\cite{Bieb}):

\begin{thrm}[Bieberbach]\label{Bieberbach}
A finitely generated group $\Gamma$ can be realised as a crystallographic group in $\operatorname{Iso}(\R^n)$ if and only if $\Gamma$ contains a normal free abelian subgroup of finite index. This subgroup is of rank $n$ and  acts on $\R^n$ by translations. For a given $n$ there exists only a finite number of such groups $\Gamma \subset \operatorname{Iso}(\R^n)$ up to conjugation.
\end{thrm}

\rmk The finiteness result in fact gives a partial solution for Hilbert's 18th problem, which originally motivated Bieberbach in his studies. The finite group $G = \Gamma/\Z^n$ is often called {\it the rotation group} of $\Gamma$.

\begin{lemma}\label{geometric Bieberbach}
Let $M$ be a compact  manifold of dimension $n$. The following conditions are equivalent:
\begin{itemize}
\item[(1)] $M$ admits a Riemannian metric $g$, such that the Riemann curvature tensor $R_g$  vanishes  everywhere;
\item[(2)] $M = \R^n/\Gamma$  for some torsion-free crystallographic group $\Gamma \subset \operatorname{Iso}(\R^n)$;
\item[(3)] $M = T/G$, where $T = (S^1)^{\times n}$ is a  torus and $G$ is a finite group which acts on $T$ freely.
\end{itemize}
\end{lemma}

The proof can be founded e.g.  in \cite{W} (Theorem 3.3.1).

In this article, we are  using  the  complex version of the same theory.  Let $\operatorname{Iso}_{\C}(\C^n)$ be the group of biholomorphisms of $\C^n$ preserving the standard Hermitian metric. A {\it complex crystallographic group } is a discrete subgroup $\Gamma \subset \operatorname{Iso}_{\C}(\C^n)$, such that $\C^n/\Gamma$ is compact. 
From the Lemma~\ref{geometric Bieberbach} we easily deduce the following generalisation:

\begin{lemma}\label{complex geometric Bieberbach} 
Let $X$ be a compact complex manifold. The following conditions are equivalent:
\begin{itemize}
\item[(1)] $X$ admits a K\"ahler metric with everywhere vanishing curvature tensor;
\item[(2)] $X = \C^n/\Gamma$ for some torsion-free complex crystallographic group $\Gamma \subset \operatorname{Iso}_{\C}(\C^n)$;
\item[(3)]$X = T/G$, where $T$ is a (compact) complex torus and $G$ is a finite group which acts on $T$ freely and holomorphically.
\end{itemize}
\end{lemma}

\begin{rmk} The torsion-free condition can be omitted if we consider {\it (K\"ahler) orbifolds} instead of manifolds. 
\end{rmk}
\begin{proof}
{\it (1)} $\implies$ {\it (2)}. Let $\tilde{X} \to X$ be the universal covering. Lift a flat K\"ahler metric to $\tilde{X}$, choose local coordinates on a small ball $B \subset \tilde{X}$ and extend them globally, using the Levi-Civita connection. Since this connection is flat, it is well-defined and establishes a  biholomorphism $\tilde{X} \to \C^n$. Moreover, since the connection is orthogonal, this is an isometry.  Since $\pi_1(X)$ acts on $\tilde{X}$ by holomorphic isometries, it is complex crystallographic.

{\it(2)} $\implies$ {\it(3)} It follows from Bieberbach's theorem that $X$ admits a finite Galois covering by a smooth compact manifold $T$. Moreover, $T$ is isomorphic to $\C^n/\Lambda$ where $\Lambda$ is a free abelian group acting on $\C^n$ by holomorphic isometries. It follows that $T$ is a complex torus.

{\it (3)} $\implies$ {\it(1)} Take any flat K\"ahler metric on $T$ and average its K\"ahler form by the action of $G$. This will descend to a K\"ahler form on $X$ and the associated metric will be flat.
\end{proof}

\begin{rmk}
A manifold which satisfies one of the properties of Lemma~\ref{complex geometric Bieberbach} is sometimes called {\it K\"ahler-Bieberbach manifold}. Another frequent term for compact K\"ahler manifolds admitting flat K\"ahler metric, but not isomorphic to a complex torus, is {\it hyperelliptic manifolds}. In the context of algebraic geometry, one can also use the term {\it hyperelliptic varieties}.
\end{rmk}

\begin{rmk}
 If $X = \C^n/\Gamma = T/G$ is a flat K\"ahler manifold, then the monodromy of its flat Levi-Civita connection  is isomorphic to $G$. Moreover, if $p \colon T \to X$ is the natural projection and $x = p(0)$, the monodromy representation of $G$ on $T_xX \simeq T_0T$ is isomorphic to the representation of $G$ on $H_1(T, \R)$ via the exponential isomorphism $T_0T \simeq H_1(X, \R)$.
\end{rmk}

Let $X$ be a flat K\"ahler manifold isomorphic to  $T/G$ for a complex torus $T$ and a finite group $G$, which acts on $T$ freely. Modulo torsion the cohomology ring of $X$ is  determined by the representation $G \to \operatorname{GL}(H^1(T, \Q))$. Indeed,
\[
H^{\hdot}(X, \Q) = (H^{\hdot}(T, \Q))^G = (\Lambda^{\hdot}H^1(T, \Q))^G.
\]

\subsection{Examples.}\label{examples}
We finish this section by giving a number of examples of flat K\"ahler manifolds:
\begin{itemize}
\item[(1)] The obvious ones are just complex tori. The first non-trivial examples arise in complex dimension $2$. These are the so-called {\it bielliptic surfaces}. Here is an explicit example of a bi-elliptic surface. Take two elliptic curves $E_1$ and $E_2$. The group of automorphisms of $E_1$ preserving the origin is isomorphic to $\Z/d\Z$, where $d \in \{2, 4, 6\}$. Pick a $d$-torsion element $\tau \in E_2$. The group $\Z/d\Z$ acts on $E_1 \times E_2$ by the formula
\[
[m] \cdot (x, y) = ([m]\cdot x, y+ m\tau)
\]
(here $m$ is an integer and $[m]$ denotes its residue class in $\Z/d\Z$). This action is free and the quotient surface $S = \left (E_1 \times E_2 \right )/ (\Z/d\Z)$ admits a flat K\"ahler metric. See (\cite{BHPV}. Ch. V. 5, Example BII)  for further discussion.

\item[(2)]Observe that one can apply an analogue of the construction from the previous example for any pair of complex tori $T_1$ and $T_2$, and get   a flat K\"ahler manifold $$X:= (T_1 \times T_2)/(\Z/2\Z),$$ where the group acts as  multiplication by $(-1)$ on $T_1$ and as  a translation by a $2$-torsion element on $T_2$.

 Sometimes  $T_1$ admits a (group) automorphism of order $d>2$, and then one can use the same construction to obtain a free action of $\Z/d\Z$ on $T_1 \times T_2$. Such manifolds are particular cases of higher-dimensional analogues of bi-elliptic surfaces,  known as {\it Bagnera- de Franchis manifolds}\footnote{The bi-elliptic surfaces were initially studied by Bagnera and de Franchis in 1908, \cite{BdF}. See also an early work of Enriques and Severi \cite{ES}.} see e.g. \cite{Cat15}, Subsection 4.1. 
 
 Of course, if $T_1$ and $T_2$ are non-algebraic, then $X = (T_1 \times T_2)/(\Z/d\Z)$ is also non-algebraic. However, all constructions of this kind can lead us only to flat K\"ahler manifolds with $b_1 \neq 0$. Thus, for our purposes flat K\"ahler manifolds with trivial first cohomology are of the main interest.

\item[(3)]
 In \cite{DHS} many examples of flat K\"ahler manifolds are given. There one can also find a complete list of deformation types of $3$-dimensional flat K\"ahler manifolds with their Hodge and Betti numbers. For example,  there exists a $3$-dimensional flat K\"ahler manifold with trivial first cohomology and rotation group isomorphic to the dihedral group $D_4$ (i.e. the one of order $8$; see \cite{CD18} for the explicit construction). There is also a family of $n$-dimensional flat K\"ahler manifolds with vanishing first Betti number and rotation group isomorphic to $(\Z/2\Z)^{n-1}$\, the so-called {\it complex Hantzsche - Wendt manifolds}\footnote{By definition, the monodromy of the Levi-Civita connection on Hantzsche-Wendt manifold should also be contained in $\operatorname{SU}(n)$} (\cite{H}; \cite{DHS}, Section 4). Both Hantzsche-Wendt manifolds and flat K\"ahler $D_4$-threefolds possess no non-zero holomorphic $2$ - forms. Therefore, as we explained in  Introduction, they are always algebraic.

A priori it is unclear if non-algebraic flat K\"ahler manifolds with vanishing first Betti number exist. At the end of this paper (Section \ref{quaternionic doubles}) we produce several examples of such manifolds. We also give a general construction which starts with a flat K\"ahler manifold with $b_1=0$ and produces a new flat K\"ahler manifold, which also has $b_1=0$ and admits a (non-degenerate) holomorphic $2$-form. As it follows from Theorem \ref{deformations}, such manifolds can be deformed to non-algebraic ones.

\end{itemize}

Finally, we would like to emphasise, that   Bieberbach's theorem (Theorem \ref{Bieberbach})  together with  Lemma \ref{complex geometric Bieberbach}  implies that there exist only finitely many topological types of flat K\"ahler manifolds in every dimension.

\section{Non-algebraic deformations of flat K\"ahler manifolds with $h^{2,0} \neq 0$}\label{main section}

In this section we prove the following theorem:

\begin{thrm}\label{deformations}
Let $X$ be a compact complex manifold admitting a flat K\"ahler metric. Assume that $H^0(X, \Omega^2) \neq 0$. Then there exists a flat smooth holomorphic family $\pi \colon \mathcal{X} \to \CP^1$, such that
\begin{itemize}
\item All the fibres $X_t := \pi^{-1}(t)$ also admit flat K\"ahler metric.
 \item$\pi^{-1}(0) \simeq X$ 
 \item The set $\mathcal{R} := \{t \in \CP^1 | \ X_t \  \text{is algebraic}\} $ is at most countable.
 \end{itemize}
\end{thrm}

\subsection{Twistor spaces: classical theory}\label{classical twistors} Let us, first of all, explain the heuristics that stay behind our theorem. It comes from the theory of Hyperk\"ahler manifolds. Recall that {\it a Hyperk\"ahler manifold} is a quadruple $(M, g, I, J)$, where $(M, g)$ is a Riemannian manifold and $I$ and $J$ are integrable complex structures, which anti-commute and both are K\"ahler with respect to $g$.  Together $I$ and $J$ generate an action of the algebra of (Hamiltonian) quaternions on the tangent bundle of $M$.

If $\omega_I$ and $\omega_J$ are the $g$-Hermitian forms of $I$ and $J$, the form $\sigma_I:= \omega_I + \sqrt{-1}\omega_J$ is a holomorphic symplectic form on $(M, I)$. Vice versa, if $(M, I)$ is a complex manifold of K\"ahler type which admits a holomorphic non-degenerate $2$-form, there exist such $g$ and $J$, that $(M, g, I, J)$ is a Hyperk\"ahler manifold (see e.g. \cite{Beau}, Prop. 4).\footnote{ In the literature $(M, I)$ is often assumed to be irreducible holomorphic symplectic and simply connected. A more general version of this statement follows from Beaville-Bogomolov decomposition theorem and some equivariant techniques which we describe below.} 

A holomorphically symplectic manifold is said to be {\it irreducible holomorphic symplectic} if the space of holomorphic $2$-forms on it is generated by a symplectic form. The Beauville-Bogomolov decomposition theorem (\cite{Beau}, Th. 2) implies that any Hyperk\"ahler manifold, up to a finite covering, is isomorphic to a product of irreducible ones and a complex torus.

 Consider the set of purely imaginary quaternions of unit norm. Every such quaternion $q$ satisfies  $q^2=-1$. This naturally forms a $2$-dimensional sphere $S^2$ inside $\mathbb{H}$ which we immediately identify with $\CP^1$. For each $q \in \CP^1$ one has the corresponding almost complex structure $I_q$ on the smooth manifold $M$.  

It is straightforward to see that all $I_q$ are integrable and K\"ahler with respect to $g$. 

Consider the product $M \times \CP^1$ with the  almost complex structure $I_{\tw}$ defined by the following rule: in a point $(m, q) \in M \times \CP^1$ the operator $I_{\tw}$ is equal to $I_q \oplus J_0$, where  $J_0$ is the standard complex structure on $\CP^1$. The complex manifold $\Tw(M):= (M \times \CP^1, I_{\mathrm{tw}})$ is called {\it the twistor space} of $M$. 

\begin{thrm}\label{twistors classic}
Let $M =  (M, g, I, J)$ be a Hyperk\"ahler manifold and $I_{\tw}$  the almost complex structure operator on $\Tw(M) = M \times \CP^1$ as above.
\begin{itemize}
\item[(1)] The complex structure $I_{\tw}$ is integrable and the natural projection $\Tw(M) \xrightarrow{\pi} \CP^1$ is holomorphic. Each fibre $\pi^{-1}(q)$ is a complex manifold $M_q := (M, I_q)$ with the complex structure $I_q$ as defined above.
\item[(2)] Assume that $M$ is compact and  irreducible holomorphic symplectic. The set $$R:= \{ q \in \CP^1 | M_q \text{ is algebraic}\}$$ is at most countable.
\end{itemize}
\end{thrm}
\begin{proof}
See \cite{K} for the first statement and \cite{Ver}, Theorem 2.2. for the second. Originally the second statement appeared in \cite{Fuj} (Theorem 4.8(2)).
\end{proof}

One would like to say that  Theorem~\ref{twistors classic} immediately implies  Theorem~\ref{deformations} in case when $X$ admits a holomorphic symplectic form. Unfortunately, even if a flat K\"ahler manifold is holomorphically symplectic,  it is usually not irreducible holomorphic symplectic. This can be overruled by introducing equivariant versions of some classical linear algebraic arguments, which we do in the next subsection.

What seems to be a more serious problem is that $X$ might carry only degenerate holomorphic $2$-forms. Our plan is to take any non-zero holomorphic $2$-form on $X$ and, using the flatness assumption, imitate the  twistor family in the direction, in which this form is non-degenerate, without changing the holomorphic structure on the kernel of this form.

\subsection{Equivariant  $\C$-symplectic linear algebra.}\label{linear algebra}
Now we are going to describe some linear algebraic constructions relating complex symplectic structures to Hermitian quaternionic structures. This is rather classical, but we will also need an equivariant version of it.

Fix the following notation. Let $V$ be a real vector space with  a complex structure operator $$I \colon V \to V, \ I^2 = - 1.$$ Then the exterior powers of $V^*_{\C} := V^* \o \C$ carry the decomposition $$\Lambda^kV_{\C}^* = \bigoplus_{p+q=k}\Lambda^{p,q}V^*_{\C}.$$  Let $\sigma \in \Lambda^{2,0}V^*_{\C}$ be a complex non-degenerate form of type $(2,0)$. A {\it compatible hyper-Hermitian  structure} on $V$ is a pair $(g, J)$, where $J$ is an automorphism of $V$ satisfying $J^2=-1$ and $JI=-IJ$, and $g$ is a positive scalar product on $V$, Hermitian with respect to both $I$ and $J$. Observe that in this case the algebra generated by $I$ and $J$ is isomorphic to the algebra of Hamiltonian  quaternions.

\begin{lemma}\label{local hyperhermitian}
Assume that a finite group $G$ acts on $V$ linearly, commuting with $I$ and preserving $\sigma$. Then there exists a $G$-equivariant  compatible hyper-Hermitian structure (that is,  both $g$ and $J$ can be chosen to be $G$-invariant).
\end{lemma}
\begin{proof}
{\bf Step 1.} Denote by $\sigma_1$ the real part of $\sigma$ and let $h_1$ be any $G$-invariant  $I$-Hermitian metric on $V$. Consider the operator $A= h_1^{-1} \circ \sigma_1 \colon V \to V^* \to V$.  In other words, $A$ is defined by the property
\[
h_1(Ax, -) = \sigma_1(x, -).
\]
Of course, $A$  is invertible and commutes with every element in $G$. Observe also, that $A$ anti-commutes with $I$. Indeed,
\[
h_1(IAx, y) = -h_1(Ax, Iy) = - \sqrt{-1}\sigma_1(x, y), 
\]
while
\[
h_1(AIx, y) = \sigma_1(Ix, y) = \sqrt{-1}\sigma_1(x, y).
\]

{\bf Step 2.} 
Since $\sigma_1$ is skew-symmetric, we see that $A$ is also skew-symmetric (with respect to $h_1$), thus $A^2$ is symmetric. Moreover, $h_1(A^2x, x) = -h_1(Ax, Ax)$, so $A^2$ is negative definite. Therefore $A^2$ can be diagonalized in an $h_1$-orthonormal basis:
\[
A^2 = 
\begin{pmatrix} \alpha_1 & 0 & 0 & \ldots  &0\\
0& \alpha_2 & 0 &\ldots& 0 \\
0& 0& \alpha_3& \ldots & 0\\
0& \ldots &\ldots & \ldots & 0\\
0& \ldots & \ldots & \ldots & \alpha_n
\end{pmatrix}
\]
with $\alpha_i \in \R_{<0}$. Consider a polynomial $P \in \mathbb{R}[t]$ which satisfies $P(\alpha_i) = \frac{1}{\sqrt{-\alpha_i}}$  for $1 \le i \le n$ and denote $S:= P(A^2)$.  By construction, this is a symmetric positive operator, which commutes both with $A$ and with elements from $G$.  We also have $(AS)^2 = -\operatorname{Id}_{V}$.

Since $A$ anti-commutes with $I$, and $S$ is a polynomial of even degree in $A$, the operator $S$ commutes with $I$.

{\bf Step 3.} Define $g(x, y):= h_1(Sx, y)$ and $J:=AS$. It is clear that both $g$ and $J$ are $G$-invariant and that $J$ is an operator of an almost complex structure.

 Since $S$ commutes with $I$, the metric $g$ is $I$-Hermitian. What is left to check is that $J$ is $g$-orthogonal and anti-commutes with $I$.

The first property follows from an explicit computation, which uses the fact that $A$ is skew-symmetric and $S$ is symmetric. The second follows from
\[
\begin{gathered}
g(IJx,y) =g(Jx, -Iy) = h_1(AS^2x, -Iy) =\\
=\sigma_1(S^2x, -Iy) = -\sqrt{-1}\sigma_1(S^2x, y),
\end{gathered}
\]
while
\[
\begin{gathered}
g(JIx, y) = g(Ix, -Jy) = h_1(AS^2y, -Ix) = \\
= \sigma_1(S^2y, -Ix) = -\sigma_1(-Ix, S^2y) = \sqrt{-1}\sigma_1(S^2x, y)
\end{gathered}
\]
Here we used the fact that $\sigma_1(S^2x, y) = h_1(AS^2x, y) = h_1(Ax, S^2y) = \sigma_1(x, S^2y)$, and, what is more important, the fact that $\sigma_1$ is a real part of $(2,0)$-symplectic form with respect to $I$.
 \end{proof}
 
 Let us also prove the following linear algebraic proposition, which will play an important role in the proof of the main theorem:
 \begin{predl}\label{linear trianalytic}
 Let $U$ and $W$ be finite-dimensional real vector spaces. Assume that $I$ is a complex structure operator on $U$ and $W$ is endowed with a faithful action of the quaternionic algebra $\mathbb{H} \times W \to W$. Therefore, for every purely imaginary $q \in \mathbb{H}$ with $q^2=-1$, we obtain a complex structure operator $I_q \colon W \to W$. Let $H \subset U \oplus W$ be a vector subspace of real codimension $2$, which is a complex hyperplane in $(U \oplus W, J_q:= I \oplus I_q)$  for every $q \in \CP^1 \subset \mathbb{H}$. Then $H = W$ (and consequently $\dim_{\C} U = 1$).
 \end{predl}
 \begin{proof}
Consider $H':= H \cap W$. By the assumptions, this is an $\mathbb{H}$-submodule in $W$. Therefore, its real dimension is divisible by $4$. But $\codim_{W}H' \le \codim_{U \oplus W} H = 2$. Hence $\codim_{W}H' = 0$ and $H'=W$.Thus, $H$ contains $W$. This is possible only if $H = W$ and $\dim_{\C}U = 1$.
 
 \end{proof}

\subsection{Construction of the twistor family.}\label{construction}

Starting from now, let $X$ be a compact flat K\"ahler manifold and $p \colon T \to X$ be a Galois covering by a complex torus $T$, so that $X = T/G$ for some finite group $G$. Assume that $\eta \in H^0(X, \Omega^2)$ is a non-zero holomorphic $2$-form.  It defines a morphism of holomorphic vector bundles $\eta \colon T_X \to T^*_X$. Consider  its kernel $\mathcal{E}:= \Ker \eta \subset TX$.  A priori this is only a holomorphic subsheaf, but since $p^*\eta$ is a holomorphic $2$-form on $T$ and all holomorphic forms on a complex torus are invariant under parallel transport, it is of constant rank, hence a $\mathcal{E} \subset TX$ is a subbundle. 

Moreover, since $\eta$ is closed, $\mathcal{E}$ is involutive and defines a holomorphic foliation on $X$.

\begin{predl}\label{global hyperhermitian}
There exists a Riemannian metric $g$ on $X$, a family of complex structures $J_q$ on $X$, parametrised by $\CP^1$, and an integrable subbundle $\mathcal{F} \subset TX$, such that:
\begin{itemize}
\item[(1)] The metric $g$ is flat;
\item[(2)] all the complex structures $J_q$ are integrable and K\"ahler with respect to $g$;
\item[(3)] the subbundles $\mathcal{E}$ and  $\mathcal{F}$ are holomorphic with respect to $J_q$ for any $q$;
\item[(4)] $TX = \mathcal{E} \oplus \mathcal{F}$ as a holomorphic vector bundle and $\mathcal{F} \perp_g \mathcal{E}$.
\end{itemize}
\end{predl}
\begin{proof}
 Denote by $V$ the (real) tangent space to the origin in $T$. The complex structure on $T$ induces  a complex structure operator $I$ on $V$.  As usual, we can identify $V$ with $H_1(T, \R)$. Therefore, $V$ is naturally endowed with  a  complex-linear action of $G$. Every $G$-invariant tensor on $V$ defines a homogeneous $G$-invariant tensor on $T$, which descends to $X = T/G$.
 
 Let $\eta_0:= (p^*\eta)(0) \in \Lambda^{2,0}V^*_{\C}$ and $E =( p^*\mathcal{E})_0 = \Ker \eta_0$. Since $\eta_0$ is a $G$-invariant complex bilinear form, $E$ is a $G$-invariant subspace of $V$,and is preserved by $I$. Since $G$ is finite, we can find a $G$-invariant $I$-complex complement $F \subset V$, so that $V = E \oplus F$ as a  complex representation of $G$. The form $\eta_0$ restricts to a non-degenerate form on $F$.  By  Lemma~\ref{local hyperhermitian} there exists a $G$-invariant compatible hyper-Hermitian structure on $F$. Let $g_F$ be the corresponding (hyper-)Hermitian metric on $F$ and $(I_q)_{q \in \CP^1}$ be the associated family of complex structure operators. 
 
 Take a $G$-invariant $I$-Hermitian metric $g_E$ on $E$ and put $g_0 := g_E \oplus g_F$ and $$J_{q,0}:= I|_E \oplus I_q.$$  As  we have explained above, the metric $g_0$ induces a $T$-invariant metric  $\tilde{g}$ on $T$, which descends to  a metric $g$ on $X$. Similarly,  the family of  operators $J_{q,0}$ induce a family of $g$-orthogonal almost complex structures $J_q$ on $X$, which are obtained by descending the $G$-invariant family of homogeneous complex structures $\tilde{J_q}$ on $T$.  Finally, $F$ induces a homogeneous $G$-invariant subbundle $\widetilde{F}$ in the holomorphic tangent bundle of $T$, which descends to a subbundle $\mathcal{F}$ on $X$. This foliation is clearly preserved by every $J_q$ and its fibres are $g$-orthogonal complements to the fibres of $\mathcal{E}$ in each point.
 
 Now we need to check the properties listed above. Since all of them are local, it is sufficient to check them for the pre-images of these objects in $T$, that is for $\tilde{g}, \widetilde{J_q}. \widetilde{F}$ and $\widetilde{\mathcal{E}} = p^*\mathcal{E}$. There exists a flat torsion-free connection $\nabla$ on $T$, which preserves all $T$-invariant tensors. Therefore $\nabla \tilde{g}= 0$ and $\nabla$ is the Levi-Civita connection of  $\tilde{g}$. In particular, $\tilde{g}$ is flat.  Because $\nabla$ is torsion-free and preserves $\widetilde{\mathcal{F}}$, this subbundle is integrable. From $\nabla \widetilde{J_q} = 0$ for every $q$, we deduce {\it (2)}. The properties {\it (3)} and {\it (4)} follow from the construction.

\end{proof}

Now we are able to define $\mathcal{X}$ as the ``twistor space'' of $(X, g, J_q)$, that is, an almost complex manifold isomorphic to  $\CP^1 \times X$ as a smooth manifold, and endowed with the almost complex structure
\[
(I_{\tw})_{(q,x)} = I_0 \oplus J_q,
\]
where $I_0$ is the standard complex structure on $\CP^1$.

\begin{predl}\label{twistor integrability}
The almost complex structure $I_{\tw}$ is integrable.
\end{predl}
\begin{proof}
Since integrability of the complex structure is a local property, we can prove it for a universal cover $\widetilde{\mathcal{X}} \to \mathcal{X}$. It is isomorphic to $E \times \Tw(F)$, where $E$ and $F$ are complex vector spaces as above and $F$ is endowed with the Hyper-Hermitian structure $(g, J_q)$. The classical twistor space $\Tw(F)$ is a complex manifold by Theorem \ref{twistors classic}, \cite{K}.
\end{proof}

\begin{rmk}
In \cite{CC} Catanese and Corvaja gave a description of  the Teichm\"uller space $\mathfrak{T}(X)$ of a flat K\"ahler manifold $X$. Every family of the form $\mathcal{X} = (X, g, J_q)$ defines a  holomorphic map $\CP^1 \to \mathfrak{T}(X)$. Studying the images of  these curves might be useful for better understanding of the geometry of $\mathfrak{T}(X)$.
\end{rmk}

\subsection{$\operatorname{SU}(2)$-action.}\label{SU2}

The group generated by unitary imaginary quaternions inside $\mathbb{H}^{\times}$ is isomorphic to $\operatorname{SU}(2)$. If $X$ is a flat K\"ahler manifold endowed with a non-zero holomorphic $2$-form, from Proposition \ref{global hyperhermitian} we obtain a faithful action of the quaternionic algebra $\mathbb{H}$ on the tangent bundle of $X$, which induces an action of $\operatorname{SU}(2)$ on smooth sections of the tangent bundle. In the notations of the previous subsection, it is generated by the operators $(\operatorname{Id}_{\mathcal{E}} \oplus I_{q})_{q \in \CP^1}$. 

This action can be extended to a linear action of $\operatorname{SU}(2)$ on the space of differential forms on $X$.

\begin{predl}\label{SU2 to cohomology} 
The defined above $\operatorname{SU}(2)$-action descends to action by algebra automorphisms on cohomology. This action preserves the intersection form.
\end{predl}
\begin{proof}
Each cohomology class on $X$ can be uniquely represented by a harmonic form. The map $p \colon (T, \tilde{g}) \to (X, g)$ is a local isometry, thus harmonic forms on $X$ are the same as $G$-invariant harmonic  forms on $T$. These are the same as $G$-invariant forms preserved by the action of $T$ on itself by translations. Here we use the fact that harmonic forms for a K\"ahler metric on a torus are the same as left-invariant forms, see \cite{BL92}, Prop. 1.4.7. Such forms clearly form a $\operatorname{SU}(2)$-invariant subalgebra inside the  de~Rham algebra of $T$, hence this action descends to cohomology. 

By construction, this action preserves the metric $g$, so it also preserves the associated volume form and the Poincar\'e pairing.
\end{proof}

\begin{predl}\label{SU2 invariance}
Let $\alpha \in H^{2p}(X, \C)$.  Let $\pi \colon \mathcal{X} \to \CP^1$ be the 'twistor' family as defined above and let $X_q:= (X, J_q) = \pi^{-1}(q)$.

Then $\alpha \in H^{p,p}(X_q)$ for every $q \in \CP^1$ if and only if it is $\operatorname{SU}(2)$-invariant.
\end{predl}
\begin{proof}
Each complex structure $J_q$ generates an $\operatorname{U}(1)$-action on the tangent bundle of $X$ which descends to the cohomology (and in fact defines the Hodge decomposition on $H^{\bullet}(X, \C)$). The condition $\alpha \in H^{p,p}(X_q)$ means that $\alpha$ is invariant under this $\operatorname{U}(1)$-action and by assumption this holds  for every $q \in \CP^1$. But a vector is invariant under a representation of a compact Lie group if and only if it is invariant under every its $1$-parameter subgroup.
\end{proof}

\subsection{Non-algebraic points in the twistor family.}\label{proof of the main theorem}

Let us now prove the following proposition:
\begin{predl}\label{No ample trianalytic}
Let $D \subset X$ be an immersed submanifold of (real) codimension $2$. Assume that $D$ is holomorphic for every complex structure $J_q$. Then the divisor $[D]$ is not ample in $X_q$ for any $q$.
\end{predl}
\begin{proof}

Take any smooth point $x \in D$ and apply the Proposition~\ref{linear trianalytic} for $W:=\mathcal{F}|_x$ and $U:=\mathcal{E}|_x$  and $H:= T_xD \subset U\oplus W = T_xX$. We see that $T_xD = \mathcal{F}|_x$, hence it $D$  coincides with a leaf of the foliation, defined by $\mathcal{F}$. 

Assume that $[D] \in \operatorname{Pic}(X_q)$ is ample for some $q$. Let $g_{E}$ be the restriction of the metric $g$ on $\mathcal{E}$ and $\omega_E$ be its Hermitian form on $X_q$. This is a non-negative closed $(1,1)$-form, which is positive on $\mathcal{E}$ and its kernel coincides with $\mathcal{F}$. In particular, $\omega_E$ restricts on $D$ by $0$ and 
\[
\int_{X_q} \omega_E \cdot [D]^{n-1} = 0 
\]
But the class of $\omega_E$ is nef and $D$ cannot be ample.

\end{proof}

Now we are ready to prove the main theorem:

\begin{proof}[Proof of  the Theorem~\ref{deformations}]

Let $\mathcal{X}$ be a ``twistor space'' as above. By  the construction, this is a complex manifold which admits a holomorphic topologically  trivial fibration $\pi \colon \mathcal{X} \to \CP^1$. For every $q \in \CP^1$ the fibre $X_q = (X, J_q)$ admits a flat K\"ahler metric $g$. Moreover, $X_0  \simeq (X, I)$.

Now we want to prove, that for a very general $q \in \CP^1$ the fibre $X_q$ is not algebraic.

For every $\alpha \in H^2(X, \Q)$ define the set $$R_{\alpha} := \{q \in \CP^1 | \alpha \in H^{1,1}(X_q) \}.$$  The subsets $R_{\alpha} \subseteq \CP^1$ are algebraic subvarieties since they are the Noether-Lefschetz loci for the (non-polarizable) variation of Hodge structures $R^2 \pi_* \underline{\Z}_{\mathcal{X}}$.  This means, that for each $\alpha$ either $R_{\alpha} = \CP^1$, or $R_{\alpha}$ is finite. In the former case, as follows from Proposition~\ref{SU2 invariance}, $\alpha$ is $\operatorname{SU}(2)$-invariant.

We claim that if $R_{\alpha} = \CP^1$, then for every $q \in \CP^1$ the class $\alpha$ is not  a class of a very ample divisor on $X_q$.

Assume the opposite. Without loss of generality, we may assume that $\alpha$ is a class of a very ample divisor $D \subset X_0$. We may also assume that $D$ is smooth and irreducible (by Bertini theorem).

 Recall the Wirtinger-Federer theorem (see e.g. \cite{St}, Ch. I): If $Z$ is a smooth compact submanifold of dimension $2k$ inside a compact K\"ahler manifold $(X, g, \omega)$, then $$\int_Z \omega^{k} \le \operatorname{Vol}_g(Z),$$ and the equality holds if and only if $Z \subset X$ is complex analytic. 
 
 Let $\omega_q:= g(J_q\cdot, \cdot)$ be the K\"ahler form on $X_q$.  Let $n = \dim X_q$. For every $q \in \CP^1$ there exists an element $B(q) \in \operatorname{SU}(2)$, such that $[\omega_q ]= B(q) \cdot [\omega_0]$ in $H^2(X, \C)$. Hence
 \[
 \begin{gathered}
 \int_D \omega_q^{n-1} = \langle (\omega_q)^{n-1}, [D] \rangle = \langle B^{n-1}(q) \cdot \omega^{n-1}_0, \alpha \rangle = \\
 =\langle B^{n-1}(q) \cdot \omega^{n-1}_0, B^{n-1}(q)\cdot \alpha \rangle = \int_D \omega_0^{n-1} = \operatorname{Vol}_g(D)
 \end{gathered}
 \]
 
 We deduce that  $D$ is holomorphic with respect to every complex structure $J_q$ in our family.  But this leads to a contradiction by Proposition~\ref{No ample trianalytic}. 

Let $\operatorname{Amp}(X)$ be the set of $\alpha \in H^2(X, \Z)$ such that $\alpha$ is a class of a very ample divisor on $X_q$ for some $q \in \CP^1$. As we have shown, for each such $\alpha$ the set $R_{\alpha}$ is a proper algebraic subset of $\CP^1$, hence finite. The set
\[
\mathcal{R}:= \bigcup_{\alpha \in \operatorname{Amp}(X)} R_{\alpha}
\]
is a countable union of finite sets, hence at most countable.

Take any point $q \in \CP^1 \setminus \mathcal{R}$. The  complex manifold $X_q$ has no ample divisors and thus is non-algebraic. This finishes the proof.
\end{proof}

\begin{cor}\label{main corollary}
Let $X$ be a compact flat K\"ahler manifold. The following conditions are equivalent:
\begin{itemize}
\item There exists a smooth holomorphic family of complex manifolds over a marked complex analytic variety  $\pi \colon \mathcal{X} \to (B, 0)$ with $\pi^{-1}(0) \simeq X$, such that every neighbourhood of $0$  contains $t$ with $\pi^{-1}(t) = X_t$ non-algebraic;
\item $H^0(X, \Omega^2_X) \neq 0$.
\end{itemize}
\end{cor}
\begin{proof}
Assume that $H^0(X, \Omega_X^2) \neq 0$. One can take the  family $\pi \colon \mathcal{X} \to \CP^1$ which is constructed in Theorem \ref{deformations}. By the result of the Theorem \ref{deformations}, the set of points $t \in \CP^1$ for which $\pi^{-1}(t)$ is algebraic is at most countable, therefore every neighbourhood of $0$ contains a point from its complement.

Vice versa, assume that $H^0(X, \Omega^2_X) = 0$, but such family $\pi \colon \mathcal{X} \to (B, 0)$ exists.  Choose a neighbourhood $B' \subset B$ of $0$ such that all the fibres $\pi^{-1}(t), t \in B'$ are K\"ahler (this can be done because $\pi^{-1}(0)$ is K\"ahler and being K\"ahler is an open property). Now, the argument, which we already mentioned in the  Introduction, can be applied: take arbitrary $t  \in B'$. The K\"ahler manifold $X_t := \pi^{-1}(t)$ has the same Hodge numbers as $X$, in particular, $H^0(X_t, \Omega^2_{X_t}) = 0$. Therefore every real cohomology class  of degree $2$ is of Hodge type $(1,1)$ and $H^2(X_t, \Q)$ is dense inside $H^2(X_t, \R) = H^{1,1}(X_t) \cap H^2(X_t, \R)$.  But the set of classes of K\"ahler forms form an open cone $\mathcal{K} \subset H^{1,1}(X_t) \cap H^2(X_t, \R)$.  Therefore we can find a K\"ahler form with a rational cohomology class. By Kodaira embedding theorem this means that $X_t$ is algebraic.
\end{proof}

As a corollary we immediately obtain the following (global version of) theorem from \cite{CD+C}:

\begin{thrm}[Catanese, Demleitner]\label{Catanese-Demleitner}
Let $X$ be a flat K\"ahler manifold. Assume that $X$ is rigid (i.e. every holomorphic family of complex manifolds with special fibre isomorphic to $X$ is constant). Then $X$ is algebraic.
\end{thrm}
\begin{proof}
Assume that $X$ is non-algebraic. Then, of course, $h^{2,0}(X) \neq 0$. Consider the family $\mathcal{X} \to \CP^1$ from the Theorem \ref{deformations}. It is non-constant because $J_0$ and $J_1$ anticommute and therefore are not conjugated by a diffeomorphism of $X$. Hence, $X$  is not rigid.
\end{proof}

Observe that in the  paper \cite{CD+C} only infinitesimal deformations were considered, so our result is slightly stronger.

\section{Examples of non-algebraic flat K\"ahler manifolds with $b_1=0$}\label{quaternionic doubles}

As we mentioned in the Introduction, many examples of non-algebraic flat K\"ahler manifolds are given by complex tori of (complex) dimension greater than one. Somehow opposite to complex tori are the flat K\"ahler manifolds with vanishing first Betti number: these are the manifolds of the form $T/G$, where $T$ is a complex torus and the finite group $G$ acts on $T$ sufficiently non-trivially from the topological point of view (i.e. $H_1(T, \Q)^G = 0$). As it was pointed out to us by the reviewer, one can construct non-algebraic flat K\"ahler  manifolds with $b_1=0$ by mimicing the construction of $3$-dimensional flat K\"ahler manifolds with monodromy group $D_4$, as it is presented in \cite{CD18}.

We will explain this construction and then present  yet another one, which allows  finding many examples of holomorphically symplectic flat K\"ahler manifolds with $b_1=0$ and which we consider to be of independent interest. 

First of all, let $E$ be an elliptic curve and $S$ a non-algebraic complex torus. Let $\tau_1, \tau_2$ be two different non-trivial $2$-torsion elements of $E$ and $\sigma \in S$ be an element of order $4$. Let $T':=E \times E \times S$.
\begin{predl}
There exists a torus $T$ isogenous to $T'$ and a free holomorphic action of the dihedral group $G \simeq D_4$ on $T$ such that $b_1(T/G) = 0$. The flat K\"ahler manifold $X = T/G$ is non-algebraic.
\end{predl}
\begin{proof}

 Let $G'$ be the subgroup of holomorphic automorphisms of $T'$ generated by
\[
s \colon (x_1, x_2, y) \mapsto (x_2+\tau_1, x_1 + \tau_2, -y)
\]
and
\[
r \colon (x_1, x_2, y) \mapsto (x_2, -x_1, y+\sigma).
\]
Here $x_i$ are coordinates on the two copies of $E$ and $y$ is the coordinate on $S$. One sees that $s^2 = r^4 = \mathrm{Id}$, while 
\[
(sr)^2(x_1, x_2, y) = (x_1 + \tau_1 + \tau_2, x_2+\tau_1+\tau_2, y).
\]
Therefore the action of this group descends to $T = T'/\langle \tau_1 + \tau_2 \rangle$ and the resulting group $G$ acting on $T$ is isomorphic to  the dihedral group $D_4$. The quotient $X = T/G$ is a flat K\"ahler manifold. If we decompose $H_1(T, \Q)$ as $H_1(E, \Q) \oplus H_1(E, \Q) \oplus H_1(S, \Q)$, we see that $s$ acts on this space through $$H_1(s) = \begin{pmatrix} 0 & 1& 0 \\1& 0 & 0 \\ 0& 0& -\operatorname{Id} \end{pmatrix},$$ while $r$ acts through $$H_1(r) = \begin{pmatrix}0 & -1&0\\1&0&0\\ 0&0&\operatorname{Id} \end{pmatrix}.$$ One easily sees that $H_1(T, \Q)^G = 0$.  

 If $X$ were algebraic, then $T'$ would be, as it is a finite cover of $X$. Although, $T'$ contains a non-algebraic subtorus $S$.
\end{proof}

Now let us present  more examples of non-algebraic flat K\"ahler manifolds with vanishing $b_1$.

 We start  with the following definition:
\begin{df}\label{definition of quaternionic double}
Let $(X, I_X)$ be a smooth complex  manifold. Assume that $X$ admits a flat connection $\nabla$ on $TX$, such that $\nabla I_X = 0$ and the monodromy action of $\nabla$ preserves a lattice $\Lambda_x \subset T_xX$. Consider the union of all possible parallel transports of vectors in $\Lambda_x$ as a subset $\Lambda \subset TX$ inside the total space of the tangent bundle. The intersection of $\Lambda$ with each fibre of the tangent bundle is a lattice inside a real vector space, thus we can take a fibrewise quotient $TX/\Lambda =: X_+$. The resulting manifold is called {\it the quaternionic double} of $X$.
\end{df}

This definition appears, for example,  in the work of   Soldatenkov and Verbitsky  (\cite{SV}, Section 3.2). The name is motivated by the following property: the flat connection $\nabla$ splits the tangent bundle $TX_+$ into a direct sum of two copies of $TX$. One can consider the  three complex structures on $X_+$ written with respect to this splitting as 
\[
I := \begin{pmatrix} I_X & 0 \\ 0 & -I_X \end{pmatrix}; \ J:= \begin{pmatrix} 0 & -1 \\ 1& 0\end{pmatrix}; \ K:= \begin{pmatrix} 0& -I_X \\ -I_X & 0 \end{pmatrix}.
\]
In the same paper(\cite{SV}) Soldatenkov and Verbitsky proved (Section 3.2.) that these complex structures are integrable and satisfy quaternionic relations
\[
I^2 = J^2 = K^2 = IJK =-1.
\] Moreover, as they showed, this establishes a Hyperk\"ahler structure on $X_+$ if and only if the initial flat connection $\nabla$ on $X$ was orthogonal with respect to some Hermitian metric (equivalently, $X$ is flat K\"ahler and $\nabla$ is its Levi-Civita connection).

\begin{cor}\label{properties of quaternionic doubles}
Let $X$ be a flat K\"ahler manifold. Then there exists a flat K\"ahler manifold $X_+$ endowed with a holomorphic projection $X_+ \to X$ such that
\begin{itemize}
\item $X_+$ admits a holomorphic symplectic form and the fibres of the projection are Lagrangian tori;
\item $b_1(X_+) = 2b_1(X)$.
\end{itemize}
In particular, if $X$ is a flat K\"ahler manifold with $b_1=0$, then $X_+$ is a flat K\"ahler manifold with $b_1=0$, which admits a non-algebraic deformation.
\end{cor}
\begin{proof}
Let $T$ be a torus, which covers $X$, so that $X = T/G$. The group $G$ acts on $H_1(T, \R)$ preserving the lattice $H_1(T, \Z)$, and this induces  a $G$-action by automorphisms on the torus $$\operatorname{Alb}(T) = H_1(T, \R)/H_1(T, \Z).$$ Of course, $H_1(T, \R)/H_1(T, \Z)$ is isomorphic to $T$, however this new action of $G$ is different from the initial one since it preserves the origin. Now the identification of $H_1(T, \R)$ with a tangent space to $T$ in the origin gives rise to an identification of $H_1(T,\R)$ with the tangent space to a point on $X$ as a $G$-module. One sees from the construction that the quaternionic double $X_+$ can be also described as $$X_+ =\frac{ T \times \left(H_1(T, \R)/H_1(T, \Z) \right)}{G}.$$ (where the action of $G$ on $T \times \left(H_1(T, \R)/H_1(T, \Z) \right )$ is diagonal). Therefore, $$b_1(X_+) = \dim (H_1(T, \R) \oplus H_1(T, \R))^G = 2b_1(X).$$
The existence of non-algebraic deformations follows from Theorem \ref{deformations}.
\end{proof}

\rmk Instead of taking the tangent bundle in the definition of quaternionic double one might take the {\bf cotangent} bundle and obtain another flat K\"ahler manifold $X^+$ (the {\it co-quaternionic double}). It also naturally carries a holomorphic symplectic form, which descends from the total space of the cotangent bundle, and also has $b_1(X^+) = 2b_1(X)$. Moreover, manifolds $X^+$ and $X_+$ admit holomorphic Lagrangian fibrations over the same base $X$ with fibres being dual complex tori.

The quaternionic and co-quaternionic doubles  might serve as a ''toy version'' for  several important examples where a similar configuration occurs, for example, a pair of Hitchin fibrations on the moduli space of $G$-Higgs bundles and the moduli space of the Langlands - dual $^LG$-Higgs bundles (see e.g. \cite{Hit}).

\end{document}